\documentclass[12pt,oneside,reqno]{amsart}
\usepackage{graphicx}
\usepackage{indentfirst,csquotes}
\usepackage{tikz-cd}

\usepackage[margin=3cm]{geometry} 
\usepackage{bm} 
\usepackage{amssymb,amsthm,amsmath}
\usepackage{xcolor,paralist,hyperref,
fancyhdr,etoolbox}
\usepackage{hyperref}

\newtheorem{theorem}{Theorem}[section] 
\newtheorem{thmx}{Theorem}

\newtheorem{lemma}[theorem]{Lemma}
\newtheorem{proposition}[theorem]{Proposition}
\newtheorem{corollary}[theorem]{Corollary}

\theoremstyle{definition} 
\newtheorem{example}[theorem]{Example}
\newtheorem{definition}[theorem]{Definition}

\newtheorem{lemmadefinition}[theorem]{Lemma-Definition}

\numberwithin{equation}{section} 

\hypersetup{ colorlinks=true, linkcolor=black, filecolor=black, urlcolor=black }

\usepackage{mathrsfs}
\usepackage{mathtools}
\usepackage{mathdots}
\usepackage{stmaryrd}
\usepackage[all,cmtip]{xy}
\usepackage{stackengine} 

\newtheorem{remark}[theorem]{Remark}

\def\N{\operatorname{\mathbb{N}}}
\def\Z{\operatorname{\mathbb{Z}}}
\def\dim{\operatorname{dim}}
\def\div{\operatorname{Div}}

\def\gr{\operatorname{Gr}}
\def\fpt{\operatorname{fpt}}

\def\rank{\operatorname{rank}}
\def\dist{\operatorname{dist}}

\begin{document}
\title{The {$F$}-pure threshold of a Schubert cycle} 
\author{Justin Fong and Mitsuhiro Miyazaki}
\date{}
\address{Department of Mathematics, Purdue University, West Lafayette, IN
47907, USA}
\email{jafong1@gmail.com}
\address{Osaka Central Advanced Mathematical Institute,
Osaka Metropolitan University, 
3-3-138 Sugimoto, Sumiyoshi-ku Osaka 558-8585, Japan}
\email{mmyzk7@gmail.com}

\let\thefootnote\relax

\begin{abstract}
The $F$-pure threshold is the characteristic $p$ counter part of the log canonical threshold in characteristic zero. It is a numerical invariant associated to the singularities of a variety, hence computing its value is important. We give a closed formula for the $F$-pure threshold of the irrelevant maximal ideal of Schubert cycles, which are the homogeneous coordinate rings of Schubert subvarieties of a Grassmannian. The main point of the computation is to give an explicit formula for the $a$-invariant of a Schubert cycle. The derivation of both formulas is made possible through the combinatorics of the underlying poset of these rings.
\end{abstract} 

\maketitle

\section{Introduction}

The $F$-pure threshold is a numerical invariant associated to rings with $F$-pure singularities. $F$-pure singularities are an example class of $F$-singularities, which are classes of singularities defined in positive characteristic using the Frobenius map. It was first introduced in \cite{MR2097584} as the positive characteristic counterpart of the \emph{log canonical threshold} of an ideal in characteristic zero, which is a measurement of a singularity on the variety defined by the ideal. Although their definitions differ, many important connections have been established between these two invariants, one of which is that the log canonical threshold is approximated by the $F$-pure threshold through reduction modulo $p$ \cite{MR2185754}. 

The $F$-pure threshold is defined as follows: Let $\N=\{0,1,2\dots\}$, and $S=\Bbbk[x_1,\dots,x_n]$ be a polynomial ring over an $F$-finite field $\Bbbk$ of characteristic $p>0$ with irrelevant maximal  ideal $\mathfrak{m}_S=(x_1,\dots,x_n)$, and $I$ be a homogeneous ideal such that $R=S/I$ is an $F$-pure domain. For any ideal $J$ define its $p^e$-th bracket power $J^{[p^e]}=(j^{p^e}: j\in J)$. Given a homogeneous ideal $\mathfrak{a}$ of $R$ with preimage $J$ in $S$, for each $e\in\N$, define the integer 
\[\nu_e(\mathfrak{a}) = \max\left\{r\in\N : J^r(I^{[p^e]}:_SI)\nsubseteq \mathfrak{m}_S^{[p^e]} \right\}.\]

\begin{definition}
The \emph{$F$-pure threshold} of $\mathfrak{a}\subset R$ is defined to be the limit
\[\fpt(\mathfrak{a})=\lim_{e\to\infty}\frac{\nu_e(\mathfrak{a})}{p^e}.\] 
\end{definition}

When $\mathfrak{a}=\mathfrak{m}$ is the irrelevant maximal  ideal of $R$, we simply call $\fpt(\mathfrak{m})$ the \emph{$F$-pure threshold of $R$}. The values of $\fpt(\mathfrak{m})$ are relatively unknown, especially when $R$ is not Gorenstein, since when $R$ is Gorenstein, $\fpt(\mathfrak{m})=-a(R)$ (where $a(R)$ is the $a$-invariant of $R$), although successful computations of $\fpt(\mathfrak{m})$ have been done in \cite{MR2097584}, \cite{MR2549545}, \cite{MR2730291}, \cite{MR3354064}, \cite{MR3719471}, \cite{MR3836659}. In this paper, we give a formula for the $F$-pure threshold of the homogeneous coordinate ring of a Schubert variety defined by $\gamma$ inside a Grassmannian, shortened to \emph{Schubert cycle}, and denoted by $G_\gamma$ where $\gamma$ is a certain tuple of integers. We also give a closed formula for the $a$-invariant of $G_\gamma$. The main results of the paper are as follows.

\begin{thmx}[Theorem \ref{thm: F-pure threshold}, Corollary \ref{cor: a-invariant schubert}]\label{thm: main (restated)}
    Let $\gamma=[a_1,\dots,a_d]$ be a tuple of integers with $1\le a_1<\cdots<a_d\le n$, $[a_1, \ldots, a_d]\neq[n-d+1, \ldots, n]$, and let $\gamma=[\beta_0,\dots,\beta_s]$ be its decomposition into blocks of consecutive integers, where each block $\beta_i$ is followed by a gap $\chi_i$. For the homogeneous maximal ideal $\mathfrak{m}$ of $G_\gamma$, one has
    \begin{align*}
        \fpt(\mathfrak{m}) &= \min\left\{\sum_{j=0}^i |\beta_j|+\sum_{j=i}^t |\chi_j| : 0\leq i\leq t\right\}, \\
        -a(G_\gamma) &= \max\left\{\sum_{j=0}^i |\beta_j|+\sum_{j=i}^t |\chi_j| : 0\leq i\leq t\right\},
    \end{align*}
    where $t=s$ if $a_d<n$, and $t=s-1$ if $a_d=n$. In the equations above, $|\beta_j|$ and $|\chi_{j-1}|$ are the sizes of the blocks and gaps. 
\end{thmx}

\section*{Acknowledgments}
The first author was partially supported by NSF Grant DMS-2100288. He would like to thank Aldo Conca, Matteo Varbaro, Linquan Ma, Vaibhav Pandey, and Uli Walther for helpful discussions. 

\section{Background}

In this section we introduce the necessary background and notations that are present throughout this paper. Let $R$ be an $\N$-graded Noetherian normal domain and $S$ denote the set of nonzero homogeneous elements of $R$. Then $S^{-1}R$ is a $\Z$-graded ring. Let $I$ be a graded fractional ideal in $S^{-1}R$. Then $I$ is said to be \emph{divisorial} if $I=R:_{S^{-1}R}(R:_{S^{-1}R} I)$, i.e., $I$ is reflexive as an $R$-module. The set of divisorial ideals of $R$, denoted $\div(R)$, form a group by the operation $I\cdot J=R:_{S^{-1}R}(R:_{S^{-1}R} IJ)$.
For $n\in\Z$ we denote $I^{(n)}$ as the $n$-th power of $I$ in $\div(R)$. Note that if $I\subsetneq R$ and $n>0$, then $I^{(n)}$ is the $n$-th symbolic power of $I$. If moreover $R$ is Cohen-Macaulay and $R_0$ is a field, then the graded canonical module $\omega$ can be realized as a graded fractionary ideal of $R$. Since $\omega$ is reflexive, $\omega$ is divisorial. We call $\omega^{(-1)}$ the \emph{anticanonical ideal}. 

The main theorem that we use to compute $F$-pure thresholds in this paper is the following:

\begin{theorem}\cite[Theorem 4.1]{MR3719471}\label{thm: fpt as limit}
    Let $S$ be an $n$-dimensional standard graded polynomial ring over an $F$-finite field $\Bbbk$ of characteristic $p>0$. Let $I$ be a homogeneous ideal such that $R=S/I$ is an $F$-pure normal domain, and let $\omega$ be the graded canonical module of $R$. For each $e\in\N$, suppose $\omega^{(1-p^e)}$ is generated in degrees $-d_1^{(e)}<\cdots<-d_{r_e}^{(e)}$. Then $\nu_e(\mathfrak{m})\in\{d_1^{(e)},\dots,d_{r_e}^{(e)}\}$.
\end{theorem}

\subsection{Algebras with Straightening Laws}

For the definition and basic facts of \emph{algebras with straightening laws} (ASL for short), we consult \cite{MR680936} and \cite{MR986492}. However, we use the term algebra with straightening law to refer to what \cite{MR680936} calls the ordinal Hodge algebra. By \cite[Proposition 1.1]{MR680936}, we see the following.

\begin{lemma}\label{lem: std-rep-by-std-rel}
    Let $A$ be a graded ASL on $\Pi$ over $B$. If $x\in A$ is represented as a polynomial of elements in $\Pi$ with coefficients in $B$, x=$f(\xi_1,\dots,\xi_n)$, then the standard representation of $x$ is obtained by repeated application of the straightening relations of $A$. 
\end{lemma}

Let $\Pi$ be a poset and $\Omega$ a subset of $\Pi$.
$\Omega$ is called a \emph{poset ideal} of $\Pi$ if whenever $\alpha\in\Omega$, $\beta\in \Pi$, $\beta<\alpha$, one has $\beta\in\Omega$.
We see by \cite[Proposition 1.2]{MR680936} the following.

\begin{lemma}\label{lem: sub-ASL}
    Let $A$ be a graded ASL on $\Pi$ over $B$ and $\Omega$ a poset ideal of $\Pi$. Set $I=A\Omega$. Then $A/I$ naturally has a graded ASL structure on $\Pi\setminus\Omega$ over $B$.
\end{lemma}

\begin{definition}
    Let $A$ be a graded ASL on a poset $\Pi$ over a commutative ring $B$. Suppose $\Omega$ is a poset ideal in $\Pi$, and let $I=A\Omega$. We say that $I$ (or $\Omega$) is \emph{straightening-closed} if every standard monomial appearing in the staightening relation of $\delta\tau$ for any $\delta,\tau\in\Omega$ incomparable in $\Pi$ have at least two factors in $\Omega$.
\end{definition}

\begin{lemma}\label{lem: straightening-closed}
    Suppose $A$ is a graded ASL on a poset $\Pi$ over a commutative ring $B$, $\Omega$ a straightening-closed poset ideal of $\Pi$, and $I=A\Omega$. Let $u\in\N$. If $z$ is an arbitrary element of $A$, then $z\in I^u$ if and only if every standard monomial appearing in the standard representation of $z$ contains at least $u$ factors in $\Omega$. 
\end{lemma}

\begin{proof}
    Suppose $\alpha,\beta\in\Pi$ are incomparable, and let $\alpha\beta=\sum_\mu c_\mu\mu$ be the straightening relation. If one of $\alpha$ or $\beta$ belongs to $\Omega$, then every standard monomial $\mu$ appearing in the straightening relation of $\alpha\beta$ contains a factor from $\Omega$, since each $\mu$ contains a factor $\zeta\in\Pi$ with $\zeta\le\alpha$ and $\zeta\le\beta$, and $\Omega$ is a poset ideal. If both $\alpha$ and $\beta$ belong to $\Omega$, then by the straightening-closed property of $\Omega$, each standard monomial $\mu$ in the straightening relation of $\alpha\beta$ contains at least two factors from $\Omega$. 
    
    By Lemma \ref{lem: std-rep-by-std-rel}, the standard representation of an element in $A$ can be obtained by repeated use of the straightening relations. Therefore, given a monomial in $\Pi$ with at least $u$ factors in $\Omega$, every standard monomial appearing in its standard representation has at least $u$ factors in $\Omega$. 
\end{proof}

\begin{remark}
    The statement of Lemma \ref{lem: straightening-closed} is equivalent to the fact that $I^u$ is generated as a $B$-module by the standard monomials containing at least $u$ factors in $\Omega$. In particular, as an ideal of $A$, $I^u$ is generated by the standard monomials of length $u$ in $\Omega$.
\end{remark}

\begin{proposition}\label{prop: same gen degree}
    Suppose $A$ is a graded ASL on a poset $\Pi$ over a commutative ring $B$. Let $\Omega_1,\dots,\Omega_n$ be straightening-closed poset ideals of $\Pi$, and set $I_i=A\Omega_i$ for $i=1,\dots,n$. For any $u_1,\dots,u_n\in\N$, the ideal $\bigcap_{i=1}^nI_i^{u_i}$ of $A$ is generated by
    \begin{equation}\label{equ: standard monomial set}
        \{\xi_1\xi_2\cdots\xi_u  : \xi_1\le\xi_2\le\dots\le\xi_u,\
        \xi_1,\xi_2,\dots,\xi_{u_i}\in\Omega_i \ \text{for $1\le i\le n$}\}
    \end{equation}
    where $u=\max\{u_1,\dots,u_n\}$. In particular, if $A$ is a homogeneous ASL, then $\bigcap_{i=1}^nI_i^{u_i}$ is generated in degree $u$. 
\end{proposition}

\begin{proof}
    We show that $I_1^{u_1}\cap\cdots\cap I_n^{u_n}$ is generated by the set of standard monomials (\ref{equ: standard monomial set}) as an ideal in $A$. It is clear that the set (\ref{equ: standard monomial set}) is contained in $I_1^{u_1}\cap\cdots\cap I_n^{u_n}$.

    Let $x\in A$ be an arbitrary non-zero element and
    \[x=\sum_j b_j\xi_{j1}\xi_{j2}\cdots \xi_{j r_j} \quad (\xi_{j1}\leq \xi_{j2}\leq\cdots \leq \xi_{j r_j})\]
    be the standard representation of $x$. By Lemma \ref{lem: straightening-closed}, for each $1\le i\le n$, one has $x\in I_i^{u_i}$ if and only if $\xi_{j1},\xi_{j2},\dots,\xi_{j u_i}\in \Omega_i$ for any $j$. Thus, $I_1^{u_1}\cap\cdots\cap I_n^{u_n}$ is contained in the ideal generated by the set (\ref{equ: standard monomial set}). 
\end{proof}

\subsection{Schubert Cycles}\label{sub-sec: schubert cycles}

Suppose $\Bbbk$ is a field of arbitrary characteristic. Let $X=(x_{ij})$ be a generic $d\times n$ matrix with $1\le d\le n$. The homogeneous coordinate ring of the Grassmannian $\gr(d,V)$ of $d$-dimensional subspaces of $V$, where $V$ is an $n$-dimensional vector space over $\Bbbk$, with respect to the Pl\"{u}cker embedding, is the $\Bbbk$-algebra $G(X)$ generated by the maximal minors of $X$. The maximal minors are identified with a tuple of their column indices $[a_1,\dots,a_d]$, hence the set of maximal minors of $X$ is 
\begin{equation}
    \Gamma(X)=\{[a_1,\dots,a_d] : 1\le a_1<\cdots<a_d\le n\},
\end{equation}
which is a poset with respect to the partial order
\[[a_1,\dots,a_d]\ge[b_1,\dots,b_d] \ \Leftrightarrow \ a_i\ge b_i \ \text{for all $i$}.\]
It is easily verified that $\Gamma(X)$ is a distributive lattice. 

Let $0=V_0\subset V_1\subset\cdots\subset V_n=V$ be a flag of subspaces of $V$ with $\dim_{\Bbbk}V_i=i$. For a fixed $\delta=[b_1,\dots,b_d]\in\Gamma(X)$, the corresponding \emph{Schubert variety} inside the Grassmannian $\gr(d,V)$ is the subvariety
\[\Omega(\delta) = \{W\in \gr(d,V) : \dim_{\Bbbk}(W\cap V_{b_i})\ge i,\ \text{for $i=1,\dots,d$}\}.\]

\begin{lemmadefinition}
Fix $\gamma=[a_1,\dots,a_d]\in\Gamma(X)$ and set $b_i=n-a_{d-i+1}+1$ for $i=1,\dots,d$. Let $\tilde{\gamma}=[b_1,\dots,b_d]$. The homogeneous coordinate ring of the Schubert subvariety $\Omega(\tilde{\gamma})$ of $\gr(d,V)$ (with respect to the Pl\"{u}cker embedding), called a \emph{Schubert cycle}, is 
\begin{equation}
    G(X;\gamma) = G(X)/J(X;\gamma),
\end{equation}
where $J(X;\gamma) = \{\delta\in\Gamma(X): \delta\ngeq\gamma\}G(X)$ is the ideal generated by the poset ideal $\{\delta\in\Gamma(X): \delta\ngeq\gamma\}$ in $G(X)$. 
\end{lemmadefinition}

By Lemma \ref{lem: sub-ASL}, we see that $G(X;\gamma)$ is a graded ASL on $\Gamma(X)\setminus\{\delta\in\Gamma(X): \delta\ngeq\gamma\}$. Note that
\begin{equation}
    \Gamma(X)\setminus\{\delta\in\Gamma(X): \delta\ngeq\gamma\}=\{\delta\in\Gamma(X): \delta\ge\gamma\}.
\end{equation}
We denote this set by $\Gamma(X;\gamma)$. The Schubert cycle $G(X;\gamma)$ is a $\Bbbk$-algebra generated by $\Gamma(X;\gamma)$. Both $G(X)$ and $G(X;\gamma)$ are ASLs on the posets $\Gamma(X)$ and $\Gamma(X;\gamma)$ respectively over $\Bbbk$.

Throughout this article, we denote $G(X;\gamma)$ by $G_\gamma$. Over an algebraically closed field of positive characteristic, it has been proven in \cite{MR799251} that Schubert varieties are $F$-split, hence the ring $G_\gamma$ is $F$-pure over such a field. It follows that if $G_\gamma$ is defined over an arbitrary field $\Bbbk$ of positive characteristic, one has that $G_\gamma\otimes_\Bbbk \overline{\Bbbk}$ is $F$-pure, where $\overline{\Bbbk}$ is the algebraic closure of $\Bbbk$. Since $F$-purity decends under pure maps, it follows that $G_\gamma$ is $F$-pure over $\Bbbk$. Therefore, the notion of $F$-pure threshold is defined for $G_\gamma$.

Given $\gamma=[a_1,\dots,a_d]\in\Gamma(X)$, one may express 
\[\gamma=[\beta_0,\dots,\beta_s], \enspace \text{where}\enspace \beta_i = a_{k_i+1},\ a_{k_i+1}+1,\dots,a_{k_{i+1}}-1,\ a_{k_{i+1}}\]
for $0\le i\le s$, as a sequence of blocks of consecutive integers, where $k_0=0$. 
We set $t=s$ if $a_d<n$, and $t=s-1$ if $a_d=n$. After each block $\beta_i$ there is a \emph{gap}
\[\chi_i = a_{k_{i+1}}+1,\dots,a_{k_{i+1}+1}-1,\]
which is a block of consecutive integers \emph{not} appearing in $\gamma$, where $a_{d+1}=n+1$. When $a_d=n$, the last gap $\chi_s$ is empty. We remark that the sizes of the blocks and gaps of $\gamma$ can be used to characterize the Gorenstein property of $G_\gamma$, that is, $G_\gamma$ is Gorenstein if and only if $|\chi_{i-1}|=|\beta_i|$ for all $i=1,\dots,t$ \cite[Corollary 8.13]{MR986492}.

If $\Pi$ is an arbitrary poset, and $x,y\in\Pi$, then $y$ is a \emph{cover} of $x$, denoted $x\lessdot y$, if $x<y$ and there is no $z\in\Pi$ such that $x<z<y$. An \emph{upper neighbor} $\zeta$ of $\gamma$ is an element of $\Gamma(X)$ which covers $\gamma$. Upper neighbors of $\gamma$ are 
\[\zeta_i = [\beta_0,\dots,\tilde{\beta}_i,\dots,\beta_s],\quad \tilde{\beta}_i = a_{k_i+1},\ a_{k_i+1}+1,\dots,a_{k_{i+1}}-1,\ a_{k_{i+1}}+1\]
for $0\le i\le t$. Note that $\tilde{\beta}_i$ is obtained by increasing the last entry of $\beta_i$ by 1.

Set the following  
\begin{align*}
    J(X;\zeta_i) &= \{\delta\in\Gamma(X) : \delta\ngeq\zeta_i\}G(X), \\
    J_i &= J(X;\zeta_i)/J(X;\gamma), \\
    \Gamma(X;\zeta_i)&=\{\delta\in\Gamma(X) : \delta\geq\zeta_i\}, \\
    \Omega_i &= \Gamma(X;\gamma)\setminus\Gamma(X;\zeta_i)
\end{align*}
for $0\le i\le t$. 
Then $J_i=\Omega_iG_\gamma$ and are distinct height 1 prime ideals for $0\le i\le t$. 
Further, by \cite[Propositions 5.2, 5.10]{MR986492} it holds that 
\begin{equation}\label{equ: principal intersection}
    \gamma G_\gamma=J_1\cap\cdots\cap J_t.
\end{equation}
Further, by \cite[Corollary 9.18]{MR986492} we see that $J_i^{(m)}=J_i^m$ for any positive integer $m$ and $0\le i\le t$. 
Moreover, by \cite[Lemma 9.1]{MR986492}, we see the following.

\begin{lemma}\label{lem: straightening-closed-schubert}
    $\Omega_i$ is straightening closed for $0\le i\le t$.
\end{lemma}

\subsection{Summary for computing $F$-pure threshold}
We outline the computation of the $F$-pure threshold of $G_\gamma$ in the list below.
\begin{itemize}
    \item [1.] Set the following,
        \begin{align}
            \kappa_i &= \sum_{j=0}^i|\beta_j|+\sum_{j=i}^t|\chi_j| \enspace\text{for $0\le i\le t$}, \label{equ: kappa numbers} \\
            \kappa &= \max\{\kappa_i : 0\le i\le t\}, \\
            \kappa' &= \min\{\kappa_i : 0\le i\le t\}, \\
            I &= \bigcap_{i=0}^t J_i^{(\kappa_i)}.
        \end{align}
    It follows from \cite{MR986492} that $I$ is isomorphic to the canonical module of $G_\gamma$ if one ignores grading.
    \item[2.] Compute $a(G_\gamma)$ by \cite{MR1188581} and show 
            $-a(G_\gamma)=\kappa$. As a consequence, one sees that $I$ is isomorphic to the canonical module of $G_\gamma$ as graded modules. 
            
    \item[3.] For $e>0$, it holds that
            $\gamma^{\kappa(p^e-1)}\omega^{(1-p^e)}=\bigcap_{i=0}^t J_i^{((\kappa-\kappa_i)(p^e-1))}$.
            In particular, $\gamma^{\kappa(p^e-1)}\omega^{(1-p^e)}$ is generated in degree 
            $(\kappa-\kappa')(p^e-1)$, and hence $\omega^{(1-p^e)}$ is generated in degree $\kappa'(1-p^e)$.
            
    \item[4.] It follows $\fpt(\mathfrak{m})=\kappa'$.
\end{itemize}

In the remainder of the article, we go through the details of each point.


\section{The \texorpdfstring{$a$}{a}-invariant of Schubert Cycles}\label{sec: a-invariant}

In this section we compute the $a$-invariant of Schubert cycles $G_\gamma$ by using a formula of Bruns and Herzog in \cite{MR1188581} that computes the $a$-invariant of an ASL. The \emph{$a$-invariant} of a positively graded Cohen-Macaulay ring $R$ over a field $\Bbbk$ with canonical module $\omega_R$ is defined to be
\[a(R) = -\min\{j : (\omega_R)_j \ne 0\}.\]
See \cite{MR494707}.

The \emph{principal chain} of a distributive lattice $\Pi$ is a chain $\mathcal{P}(\Pi)=\xi_1,\dots,\xi_m$ formed inductively as follows: $\xi_1$ is the unique minimal element of $\Pi$. If $\xi_1,\xi_2,\dots,\xi_i$ is defined and $\xi_i$ is not the unique maximal element of $\Pi$,
\[\xi_{i+1}=\nu_1\sqcup\cdots\sqcup\nu_v,\]
where $\nu_1,\dots,\nu_v$ are elements which cover $\xi_i$, and $\sqcup$ denotes the join in $\Pi$. If $\xi_i$ is the maximal element of $\Pi$, then the inductive construction terminates and $\xi_1,\dots,\xi_i$ is the principal chain. By using the principal chain, the $a$-invariant of an ASL is given as follows.

\begin{theorem}\cite[Theorem 1.1]{MR1188581}\label{thm: ASL a-invariant}
Let $R$ be a graded ASL on a distributive lattice $\Pi$ with principal chain $\mathcal{P}(\Pi)=\xi_1,\dots,\xi_m$. Then its $a$-invariant is given by
\begin{equation}\label{equ: a-invariant of schubert cycle}
    a(R) = -\sum_{i=1}^m\deg\xi_i.
\end{equation}
\end{theorem}

Since we assume $G_\gamma$ has the standard grading, each $\deg\xi_i=1$, so (\ref{equ: a-invariant of schubert cycle}) becomes
\[a(G_\gamma) = -m = -\big|\mathcal{P}(\Gamma(X;\gamma))\big|,\]
where $\big|\mathcal{P}(\Gamma(X;\gamma))\big|$ is the cardinality of the principal chain of $\Gamma(X;\gamma)$. Thus, computing the $a$-invariant becomes a combinatorial problem in which we count the number of elements making up the chain $\mathcal{P}(\Gamma(X;\gamma))$. 

\subsection*{Construction of the Principal Chain}
Let $\gamma=[a_1,\dots,a_d]\in \Gamma(X)$, and $\gamma=[\beta_0,\dots,\beta_s]$ be its expression as a sequence of consecutive blocks in $\Gamma(X;\gamma)$. The join $\zeta\sqcup\eta$ of two elements $\zeta,\eta\in \Gamma(X)$ is by taking the maximum of their corresponding entries. We describe the algorithm that constructs the principal chain of $\mathcal{P}(\Gamma(X;\gamma))=\xi_1,\dots,\xi_m$. 
\begin{itemize}
    \item The only minimal element of $\Gamma(X;\gamma)$ is $\gamma=[a_1,\dots,a_d]$, hence $\xi_1=\gamma$ is the first element of $\mathcal{P}(\Gamma(X;\gamma))$.

    \item Suppose that $\xi_1,\xi_2,\dots,\xi_i$ is defined and $\xi_i\neq[n-d+1,\ldots, n]$. 
    Let $\xi_i=[\nu_0^{(i)},\nu_1^{(i)},\dots,\nu_{s_i}^{(i)}]$ be the block decomposition of $\xi_i$ and set $t_i=s_i$ if the last entry of $\nu_{s_i}^{(i)}$ is not $n$ and $t_i=s_i-1$ otherwise. Increase the last entry of each $\nu_j^{(i)}$ by 1 for $0\le j\le t_i$.

    \item The process stops after a finite number of steps when one reaches 
    \[[n-d+1,\dots,n-1,\ n],\]
    which is the only maximal element of $\Gamma(X;\gamma)$. Thus, $\xi_m=[n-d+1,\dots,n-1,\ n]$ is always the last element of $\mathcal{P}(\Gamma(X;\gamma))$.
\end{itemize}

\begin{example}\label{ex: principal chain}
    Suppose $X$ is a $6\times11$ generic matrix (so $d=6,\ n=11$), and let
    $\gamma=[2,3,4,6,8,9]\in \Gamma(X)$. The principal chain $\mathcal{P}(\Gamma(X;\gamma))$ is:
    \begin{align*}
\gamma = \xi_1 &= [2,\ 3,\ 4,\ 6,\ 8,\ 9], \\
        \xi_2 &= [2,\ 3,\ \bm{5},\ \bm{7},\ 8,\ \bm{10}],\\
        \xi_3 &= [2,\ \bm{4},\ \bm{6},\ 7,\ \bm{9},\ \bm{11}],\\
        \xi_4 &= [\bm{3},\ \bm{5},\ 6,\ \bm{8},\ \bm{10},\ 11],\\
        \xi_5 &= [\bm{4},\ 5,\ \bm{7},\ \bm{9},\ 10,\ 11],\\
        \xi_6 &= [4,\ \bm{6},\ \bm{8},\ 9,\ 10,\ 11],\\
        \xi_7 &= [\bm{5},\ \bm{7},\ 8,\ 9,\ 10,\ 11],\\
        \xi_8 &= [\bm{6},\ 7,\ 8,\ 9,\ 10,\ 11]
    \end{align*}
    (boldfaced numbers are the previous entries that increased by 1). 
\end{example}

We define a twist operation on $\Gamma(X)$. For $a=[a_1,\dots,a_d]\in \Gamma(X)$, set $b_i=n-a_{d-i+1}+1$ for $1\le i\le d$, and define $\tilde{a}=[b_1,\dots,b_d]$. It is clear that $\tilde{a}\in \Gamma(X)$ and $\tilde{\tilde{a}}=a$. Further, if $c=[n-d+1,\dots,n-1,\ n]$, then $\tilde{c}=[1,2,\dots,d]$. 
By mapping this operation, the construction process of the principal chain becomes the following algorithm. 

\subsection*{Algorithm 2}
Let $[a_1,\dots,a_d]$ be an arbitrary tuple of positive integers where $1\le a_1<\cdots< a_d$. We define the algorithm for building a chain of $d$-tuples from $[a_1,a_2,\dots,a_d]$ to $[1,2,\dots,d]$ as follows:

Set $[a_1^{(1)},\dots,a_d^{(1)}]=[a_1,\dots,a_d]$. If $[a_1^{(1)},\dots,a_d^{(1)}],\dots,[a_1^{(i)},\dots,a_d^{(i)}]$ are defined and $[a_1^{(i)},\dots,a_d^{(i)}]\neq[1,\ldots,d]$, then define $[a_1^{(i+1)},\dots,a_d^{(i+1)}]$ by
\[a_j^{(i+1)} = 
\begin{cases}
    a_j^{(i)}-1, & \text{if $j=1$ and $a_1^{(i)}>1$, or $j>1$ and $a_j^{(i)}>a_{j-1}^{(i)}+1$} \\
    a_j^{(i)}, & \text{otherwise.}
\end{cases}\]

If one starts this algorithm from $\tilde{\gamma}$, then one gets the twisted principal chain $\tilde{\xi}_1,\tilde{\xi}_2,\dots,\tilde{\xi}_m$. In particular, the length of the principal chain and that of the chain obtained by Algorithm 2 are the same.

\begin{example}
    Let $\tilde{\gamma}=[3,4,6,8,9,10]$ be the twist of $\gamma$ from Example \ref{ex: principal chain}. The construction of the chain of tuples by Algorithm 2 is:
    \begin{eqnarray*}
        \tilde{\gamma}=\tilde{\xi}_1 &=& [3,\ 4,\ 6,\ 8,\ 9,\ 10], \\
        \tilde{\xi}_2 &=& [\bm{2},\ 4,\ \bm{5},\ \bm{7},\ 9,\ 10], \\
        \tilde{\xi}_3 &=& [\bm{1},\ \bm{3},\ 5,\ \bm{6},\ \bm{8},\ 10], \\
        \tilde{\xi}_4 &=& [1,\ \bm{2},\ \bm{4},\ 6,\ \bm{7},\ \bm{9}], \\
        \tilde{\xi}_5 &=& [1,\ 2,\ \bm{3},\ \bm{5},\ 7,\ \bm{8}], \\
        \tilde{\xi}_6 &=& [1,\ 2,\ 3,\ \bm{4},\ \bm{6},\ 8], \\
        \tilde{\xi}_7 &=& [1,\ 2,\ 3,\ 4,\ \bm{5},\ \bm{7}], \\
        \tilde{\xi}_8 &=& [1,\ 2,\ 3,\ 4,\ 5,\ \bm{6}]
    \end{eqnarray*}
    (boldfaced numbers are the previous entries that decreased by 1).
\end{example}

For $a=[a_1,\dots,a_d]\in \Gamma(X)$ with $a\ne[1,2,\dots,d]$, we define $m(a,j)=a_j-2j$ for $j$ with $a_j>j$, and $m(a)=\max\{m(a,j) : a_j>j\}$. Further, we denote the length of the sequence of elements in $\Gamma(X)$ obtained by Algorithm 2 starting from $a$ by $\tau(a)$. The following fact is easily seen.

\begin{lemma}\label{lem: m-values}
    If $j<a_j=a_{j+1}-1$, then $m(a,j+1)=m(a,j)-1$. In particular, if one decomposes $a$ into blocks as $a=[\alpha_0,\alpha_1,\dots,\alpha_u]$ where $\alpha_0$ is a special block defined by $\{a_r : a_r=r\}$, which may be empty, then $m(a)$ is obtained by one of the first entries of $\alpha_i$, $1\le i\le u$, and for the other entries of $a$ with $a_j>j$, $m(a,j)<m(a)$.
\end{lemma}

Next, we show the following.

\begin{theorem}\label{thm: total count of tuples}
    For $a=[a_1,\dots,a_d]\in \Gamma(X)$ with $a\ne[1,2,\dots,d]$, it holds that 
    \[\tau(a) = m(a)+d+1.\]
\end{theorem}

\begin{proof}
    We prove by induction on $\tau(a)$. Since $a\ne[1,2,\dots,d]$, there are at least 2 elements $a$ and $[1,2,\dots,d]$ in the sequence obtained by Algorithm 2 starting from $a$. Therefore, we begin to consider the case where $\tau(a)=2$. If $\tau(a)=2$, then $a=[1,2,\dots,d-1,d+1]$. Thus $m(a)=1-d$, and $m(a)+d+1=2=\tau(a)$.

    Next, suppose that $\tau(a)>2$ and let $b=[b_1,\dots,b_d]$ be the element obtained from Algorithm 2 from $a$.
    Then, $b\neq[1,2,\ldots, d]$, $\tau(b)=\tau(a)-1$ and $\tau(b)=m(b)+d+1$ by induction hypothesis.
    Let $a=[\alpha_0,\alpha_1,\dots,\alpha_u]$ be the block decomposition of Lemma \ref{lem: m-values}, $\alpha_i=a_{r_i},\dots,a_{s_i}$ for $1\le i\le u$. Then by Lemma \ref{lem: m-values}, 
    \[m(a)=\max\{m(a,r_i) : 1\le i\le u\},\] 
    and $m(a,j)<m(a)$ for $j$ with $r_1\le j\le d$ and $j\ne r_i$ for $1\le i\le u$. We consider by separating into three cases. 
    
    $\,$

    \noindent{{\bf Case 1:} $a_{r_1}>r_1+1$.} 
    
    In this case it holds that $b_{r_1}>r_1$. Further, $m(b,j)=m(a,j)<m(a)$ for $j$ with $j\ne r_i$ for any $i$ by the first part of Lemma \ref{lem: m-values}. On the other hand, $m(b,r_i)=m(a,r_i)-1\le m(a)-1$ for any $i$ with $1\le i\le u$. Thus, $m(b)\le m(a)-1$. On the other hand,
    \begin{align*}
        m(a) &= \max\{m(a,r_i) : 1\le i\le u\} \\
        &= \max\{m(b,r_i) : 1\le i\le u\} + 1 \\
        &\le m(b) + 1.
    \end{align*}
    Therefore, we see that $m(a)=m(b)+1$ and 
    \begin{align*}
        m(a)+d+1 &= (m(b)+d+1)+1 \\
        &= \tau(b) + 1 \\
        &= \tau(a).
    \end{align*}
    
    $\,$

    \noindent{{\bf Case 2:} $a_{r_1}=r_1+1$ and $a_{r_1+1}=r_1+2$.} 

    In this case, $b_{r_1}=r_1$, $b_{r_1+1}=r_1+2$. Therefore, by the same argument as in Case 1, we see that 
    \[m(b)=\max\big(\{m(b,r_i) : 2\le i\le u\}\cup\{m(b,r_1+1)\}\big)=m(a)-1,\]
    since $m(b,r_1+1)=m(a,r_1)-1$. The remaining arguments are the same as that of Case 1.

    $\,$
    
    \noindent{{\bf Case 3:} $a_{r_1}=r_1+1$ and $a_{r_1+1}>r_1+2$.} 

    In this case $r_2=r_1+1$ and $m(a,r_2)=m(a,r_1+1)\ge m(a,r_1)$. Therefore, $m(a) = \max\{m(a,r_i) : 2\le i\le u\}$. Further, by the same argument as in Case 1, we see that $m(b) = \max\{m(b,r_i) : 2\le i\le u\}$. The remaining arguments are the same as those of Case 1.
\end{proof}

By Theorem \ref{thm: total count of tuples}, Lemma \ref{lem: m-values}, and Theorem \ref{thm: ASL a-invariant}, we see the following.

\begin{corollary}\label{cor: a-invariant schubert}
    Let $\gamma=[a_1,\dots,a_d]\in \Gamma(X)$ and $\gamma=[\beta_0,\dots,\beta_s]$ be its decomposition into blocks of consecutive integers. Let $t=s$ if $a_d<n$, and $t=s-1$ if $a_d=n$. The $a$-invariant of the associated Schubert cycle $G_\gamma$ satisfies  
    \[-a(G_\gamma)=
    \begin{cases}
        1, & \text{if $\gamma=[n-d+1,\dots,n-1,n]$,} \\
        \kappa, & \text{otherwise.}
    \end{cases}\]
\end{corollary}

\begin{proof}
    The case where $\gamma=[n-d+1,\dots,n-1,n]$ is a direct consequence of Theorem \ref{thm: ASL a-invariant}. 
    (In this case, $G_\gamma$ is isomorphic to the polynomial ring with one variable over $\Bbbk$.) 
    We assume that $\gamma\ne[n-d+1,\dots,n-1,n]$. By Theorem \ref{thm: total count of tuples}, Lemma \ref{lem: m-values}, and Theorem \ref{thm: ASL a-invariant}, we see that 
    \begin{align*}
        -a(G_\gamma) &= d+1+\max\{n-a_{k_{i+1}}+1-2(d-k_{i+1}+1) : 0\le i\le t\} \\
        &= \max\{k_{i+1} + ((n-a_{k_{i+1}})-(d-k_{i+1})) : 0\le i\le t\}.
    \end{align*}
    By the definition of block decomposition $[\beta_0,\dots,\beta_s]$ of $\gamma$, we see that
    \[k_{i+1}=\sum_{j=0}^i|\beta_i| \quad\text{and}\]
    \[(n-a_{k_{i+1}})-(d-k_{i+1}) = \left(\sum_{j=i+1}^s|\beta_j| + \sum_{j=i}^t|\chi_j|\right) - \sum_{j=i+1}^s|\beta_j|
        = \sum_{j=i}^t|\chi_j|\]
    for any $i$. Therefore, it follows that
    \[-a(G_\gamma)=\max\left\{\sum_{j=0}^i |\beta_j|+\sum_{j=i}^t |\chi_j| : 0\leq i\leq t\right\}=\kappa.\]
\end{proof}

Since $\kappa_{i-1}-\kappa_i=|\chi_{i-1}|-|\beta_i|$ for $1\le i\le t$, we see by \cite[Theorem 8.12]{MR986492} that $J_0^{(\kappa_0)}\cap\cdots\cap J_t^{(\kappa_t)}$ is isomorphic to the canonical module of $G_\gamma$ if we ignore the grading. Further by \cite[Corollary 9.18]{MR986492}, we see that $J_i^{(\kappa_i)}=J_i^{\kappa_i}$ for $0\le i\le t$. Moreover, by Proposition \ref{prop: same gen degree}, the intersection $J_0^{\kappa_0}\cap\cdots\cap J_t^{\kappa_t}$ is generated in degree $\max\{\kappa_i : 0\le i\le t\}=\kappa$. Therefore, we see the following.

\begin{proposition}\label{prop: graded canonical module}
    The graded canonical module $\omega$ of $G_\gamma$ is 
    \[J_0^{\kappa_0}\cap\cdots\cap J_t^{\kappa_t}=J_0^{(\kappa_0)}\cap\cdots\cap J_t^{(\kappa_t)}.\]
\end{proposition}

\section{\texorpdfstring{The $F$-pure Threshold}{F-pure Threshold}}

In this section we consider the degree of the symbolic powers of anticanonical modules and $F$-pure threshold of $G_\gamma$. 

Let $m$ be a positive integer. Then since $\omega=J_0^{(\kappa_0)}\cap\cdots\cap J_t^{(\kappa_t)}$, we see that $\omega^{(-m)}=J_0^{(-m\kappa_0)}\cap\cdots\cap J_t^{(-m\kappa_t)}$. On the other hand, since $\gamma G_\gamma=J_0\cap\cdots\cap J_t$ by (\ref{equ: principal intersection}), we see
\[\gamma^{m\kappa}\omega^{(-m)}=J_0^{(m(\kappa-\kappa_0))}\cap\cdots\cap J_t^{(m(\kappa-\kappa_t))}.\]
Since each $m(\kappa-\kappa_i)$ is a non-negative integer, it follows from Proposition 
\ref{prop: same gen degree} and \cite[Corollary 9.18]{MR986492} that $\gamma^{m\kappa}\omega^{(-m)}$ is generated in degree 
\[\max\{m(\kappa-\kappa_i) : 0\le i\le t\} = m(\kappa-\kappa').\]
Since $\gamma$ is an element of $G_\gamma$ of degree 1, we have the following.

\begin{proposition}\label{prop: degree of sym power canonical}
    Let $\omega$ be the graded canonical module of $G_\gamma$ and $m$ a positive integer. Then $\omega^{(-m)}$ is generated in degree $-m\kappa'$. 
\end{proposition}

By Theorem \ref{thm: fpt as limit}, we have the following.

\begin{theorem}\label{thm: F-pure threshold}
    Suppose that the underlying field of $G_\gamma$ is $F$-finite of characteristic $p$. Let $\mathfrak{m}$ be the irrelevant maximal ideal of $G_\gamma$. Then it holds that $\fpt(\mathfrak{m})=\kappa'$.
\end{theorem}

\begin{proof}
    Let $e$ be a positive integer. By Proposition \ref{prop: degree of sym power canonical}, we see that $\omega^{(1-p^e)}$ is generated in degree $\kappa'(1-p^e)$. Thus, by Theorem \ref{thm: fpt as limit} we see that $\nu_e(\mathfrak{m})=\kappa'(p^e-1)$. Therefore,
    \[\fpt(\mathfrak{m})=\lim_{e\to\infty}\frac{\nu_e(\mathfrak{m})}{p^e}=\kappa'.\]
\end{proof}

\begin{remark}
Let $S$ be a polynomial ring where $G_\gamma\subset S$. With respect to the diagonal term order on $S$, the initial algebra $\mathrm{in}(G_\gamma)$ is the Hibi ring $\mathcal{R}_\Bbbk(\Gamma(X;\gamma))$ on the poset $\Gamma(X;\gamma)$ (see \cite{MR2356307}). We observe that $G_\gamma$ and $\mathcal{R}_\Bbbk(\Gamma(X;\gamma))$ both share the same $a$-invariant and $F$-pure threshold. 

By the structure theorem of Birkhoff, $\Gamma(X;\gamma)=J(P)$ is the lattice of poset ideals of $P$, where $P$ is the poset of join irreducibles of $\Gamma(X;\gamma)$. Let $-\infty$ and $\infty$ be two new elements outside of $P$ such that $-\infty<x<\infty$ for all $x\in P$. Define $[-\infty,\infty]=\{x\in P:-\infty\le x\le\infty\}$. For an arbitrary poset $\Pi$, if $x,y\in\Pi$ such that $x\le y$, a \emph{saturated chain} from $x$ to $y$ is a sequence of elements $z_0,z_1,\dots,z_t$ of $\Pi$ such that\
\[x=z_0\lessdot z_1\lessdot\cdots\lessdot z_t=y\]
(recall that $z_i\lessdot z_{i+1}$ means $z_{i+1}$ is a cover of $z_i$). The length of this chain is $t$. Next, define
\[\dist(-\infty,\infty)=\min\{t:\text{there is a saturated chain in $P$ from $-\infty$ to $\infty$ of length $t$}\}.\]
Recalling the fact that $\mathrm{in}(G_\gamma)$ is a level and anticanonical level Hibi ring on the poset $\Gamma(X;\gamma)$, it follows from \cite[Theorem 3.11]{MR4125854} that for any $n\in\Z$, the module $\omega_{\mathrm{in}(G_\gamma)}^{(n)}$ is generated in degree $n\cdot\rank[-\infty,\infty]$ if $n>0$, and $n\cdot\dist(-\infty,\infty)$ if $n<0$. Therefore, the $a$-invariant and $F$-pure threshold of $\mathrm{in}(G_\gamma)$ are respectively
\[-a(\mathrm{in}(G_\gamma))=\rank[-\infty,\infty] \quad\text{and}\quad \fpt(\mathfrak{m})=\dist(-\infty,\infty).\]
In fact, both quantities follow from the general results of \cite{MR3354064}.
These are equal to the $a$-invariant and $F$-pure threshold of $G_\gamma$ as well.
\end{remark}

$\,$

$\,$

\bibliographystyle{alpha}
\bibliography{bibliography.bib}

\end{document}